\documentclass[twosided,reqno]{amsart}

\usepackage{amsmath}
\usepackage{amssymb}
\usepackage{amsthm}

\theoremstyle{theorem}
\newtheorem{theorem}{\scshape Theorem }[section]

\theoremstyle{definition}

\numberwithin{equation}{section}

\begin{document}

\title[Umbral calculus and special polynomials]{Umbral calculus and special polynomials}

\author{Taekyun Kim$^1$}
\address{$^1$ Department of Mathematics, Kwangwoon University, Seoul 139-701, Republic of Korea.}
\email{tkkim@kw.ac.kr}

\author{Dae San Kim$^2$}
\address{$^2$ Department of Mathematics, Sogang University, Seoul 121-742, Republic of Korea.}
\email{dskim@sogang.ac.kr}


\maketitle

\begin{abstract}
In this paper, we consider several special polynomials related to associated sequences of polynomials. Finally, we give some new and interesting identities of those polynomials arising from transfer formula for the associated sequences.
\end{abstract}

\section{Introduction}

In this paper, we assume that $\lambda\in{\mathbb{C}}$ with $\lambda \neq 1$. For $\alpha \in {\mathbb{R}}$, the Frobenius-Euler polynomials are defined by the generating function to be
\begin{equation}\label{1}
\left(\frac{1-\lambda}{e^t-\lambda}\right)^{\alpha}e^{xt}=\sum_{n=0} ^{\infty}H_n ^{(\alpha)} (x|\lambda)\frac{t^n}{n!},{\text{ (see [1,5,13,15,20,22,23])}}.
\end{equation}
In the special case, $x=0$, $H_n ^{(\alpha)}(0|\lambda)=H_n  ^{(\alpha)}(\lambda)$ are called the {\it{$n$-th Frobenius-Euler numbers}} of order $\alpha$. As is well known, the Bernoulli polynomials of order $\alpha$ are given by
\begin{equation}\label{2}
\left(\frac{t}{e^t-1}\right)^{\alpha} e^{xt}=\sum_{n=0} ^{\infty}B_n ^{(\alpha)} (x)\frac{t^n}{n!},{\text{ (see [2,3,4,6,14,15,19,21])}}.
\end{equation}
In the special case, $x=0$, $B_n ^{(\alpha)}(0)=B_n ^{(\alpha)}$ are called the {\it{$n$-th Bernoulli numbers}} of order $\alpha$.

For $n \geq 0$, the {\it{Stirling numbers of the second kind}} are defined by generating function to be
\begin{equation}\label{2}
(e^t-1)^n=n!\sum_{l=n} ^{\infty}S_2(l,n)\frac{t^l}{l!},{\text{ (see [8-12,17,18])}},
\end{equation}
and the {\it{Stirling numbers of the first kind}} are given by
\begin{equation}\label{3}
(x)_n=x(x-1)\cdots(x-n+1)=\sum_{l=0} ^n S_1(n,l)x^l, {\text{ (see [7,8,10,17,18])}}.
\end{equation}

Let ${\mathcal{F}}$ be the set of all formal power series in the variable $t$ over ${\mathbb{C}}$ with
\begin{equation}\label{4}
{\mathcal{F}}=\left\{ \left.f(t)=\sum_{k=0} ^{\infty} \frac{a_k}{k!} t^k~\right|~ a_k \in {\mathbb{C}} \right\}.
\end{equation}
Let ${\mathbb{P}}$ be the algebra of polynomials in the variable $x$ over ${\mathbb{C}}$ and ${\mathbb{P}}^{*}$ be the vector space of all linear functionals on ${\mathbb{P}}$. As a notation, the action of the linear functional $L$ on a polynomial $p(x)$ is denoted by $\left< L~|~p(x)\right>$. Let $f(t)=\sum_{k=0} ^{\infty} a_k\frac{t^k}{k!} \in {\mathcal{F}}$. Then we define the linear functional $f(t)$ on ${\mathbb{P}}$ by
\begin{equation}\label{5}
\left<f(t)|x^n \right>=a_n,~(n\geq 0), {\text{(see [10,12,16,17,18])}}.
\end{equation}
From \eqref{5}, we note that
\begin{equation}\label{6}
\left<t^k | x^n \right>=n! \delta_{n,k},~(n,k \geq 0),
\end{equation}
where $\delta_{n,k}$ is the Kronecker symbol (see \cite{08,10,11,17,18}).

Let $f_L(t)=\sum_{k=0} ^{\infty} \frac{\left<L|x^k\right>}{k!}t^k$. Then, by  \eqref{6}, we get $\left<f_L(t)|x^n\right>=\left<L|x^n\right>$. The map $L \mapsto f_L(t)$ is a vector space isomorphism from ${\mathbb{P}}^{*}$ onto ${\mathcal{F}}$. Henceforth, ${\mathcal{F}}$ thought of as both a formal power series and a linear functional. We shall call ${\mathcal{F}}$ the {\it{umbral algebra}}. The umbral calculus is the study of umbral algebra (see \cite{10,16,17,18}).

 The order $o(f(t))$ of the non-zero power series $f(t)$ is the smallest integer $k$ for which the coefficient of $t^k$ does not vanish (see \cite{10,11,12,17,18}). If $o(f(t))=1$, then $f(t)$ is called a {\it{delta series}}, and if $o(f(t))=0$, then $f(t)$ is called an {\it{invertible series}}. Let $o(f(t))=1$ and $o(g(t))=0$. Then there exists a unique sequence $S_n(x)$ of polynomials such that $\left<g(t)f(t)^k|S_n(x)\right>=n!\delta_{n,k}$ where $n,k \geq 0$.  The sequence $S_n(x)$ is called {\it{Sheffer sequence}} for $(g(t),f(t))$, which is denoted by $S_n(x)\sim (g(t),f(t))$. If $S_n(x)\sim\left(1,f(t)\right)$, then $S_n(x)$ is called the {\it{associated sequence}} for $f(t)$ (see \cite{10,16,17,18}). From \eqref{6}, we note that $\left. \left<e^{yt}\right| p(x)\right>=p(y)$.

 Let $f(t)\in{\mathcal{F}}$ and $p(x)\in{\mathbb{P}}$. Then we have
\begin{equation}\label{7}
f(t)=\sum_{k=0} ^{\infty} \frac{\left<f(t)|x^k\right>}{k!} t^k,~p(x)=\sum_{k=0} ^{\infty} \frac{\left<t^k|p(x)\right>}{k!}x^k,{\text{ (see [17,18])}}.
\end{equation}
From \eqref{8}, we can derive the following equation:
\begin{equation}\label{8}
p^{(k)}(0)=\left<t^k|p(x)\right>{\text{ and }}\left<1\left|p^{(k)}(x)\right.\right>=p^{(k)}(0),{\text{ (see [10,16,17,18])}}.
\end{equation}
for $k \geq 0$, by \eqref{8}, we easily see that $t^kp(x)=p^{(k)}(x)=\frac{d^kp(x)}{dx^k}$.

Let $S_n(x) \sim \left(g(t),f(t)\right)$. Then we see that
\begin{equation}\label{9}
\frac{1}{g({\bar{f}}(t))}e^{y{\bar{f}}(t)}=\sum_{k=0} ^{\infty} \frac{S_k(x)}{k!}t^k,{\text{ for all }}y \in {\mathbb{C}},
\end{equation}
where ${\bar{f}}(t)$ is the compositional inverse of $f(t)$ (see \cite{17,18}).

Let $p_n(x)\sim\left(1,f(t)\right)$, $q_n(x)\sim\left(1,g(t)\right)$. Then, the transfer formula for the associated sequence is given by
\begin{equation}\label{10}
q_n(x)=x\left(\frac{f(t)}{g(t)}\right)^nx^{-1}p_n(x),{\text{ (see [11,12,16,17,18])}}.
\end{equation}
For $n \geq 0$, $b \neq 0$, the {\it{Abel sequences}} are given by
\begin{equation}\label{11}
A_n(x;b)=x(x-bn)^{n-1}\sim\left(1,te^{bt}\right).
\end{equation}
In this paper, we consider several special polynomials related to associated sequences of polynomials. Finally, we give some new and interesting identities of those polynomials arising from transfer formula for the associated sequences.

\section{Umbral calculus and special polynomials}

From \eqref{1}, we note that
\begin{equation}\label{12}
H_n ^{(\alpha)} (x|\lambda)\sim\left(\left(\frac{e^t-\lambda}{1-\lambda}\right)^{\alpha},t\right).
\end{equation}
Thus, we get
\begin{equation}\label{13}
H_n ^{(\alpha)} (x|\lambda)=\left(\frac{1-\lambda}{e^t-\lambda}\right)^{\alpha}x^n.
\end{equation}
Let us assume that
\begin{equation}\label{14}
p_n(x)\sim\left(1,t(e^t-\lambda)\right),~q_n(x)\sim\left(1, \left(\frac{e^t-\lambda}{1-\lambda}\right)^at\right),~(a \neq 0).
\end{equation}
From $x^n \sim (1,t)$, \eqref{10} and \eqref{14}, we note that
\begin{equation}\label{15}
\begin{split}
p_n(x)&=x\left(\frac{t}{t(e^t-\lambda)}\right)^nx^{-1}x^n=\frac{x}{(1-\lambda)^n}\left(\frac{1-\lambda}{e^t-\lambda}\right)^nx^{n-1}\\
&=\frac{1}{(1-\lambda)^n}xH_{n-1} ^{(n)} (x|\lambda).
\end{split}
\end{equation}
and
\begin{equation}\label{16}
q_n(x)=x\left(\frac{1-\lambda}{e^t-\lambda}\right)^{na}x^{-1}x^n=xH_{n-1} ^{(an)} (x|\lambda).
\end{equation}
From \eqref{10}, \eqref{14}, \eqref{15} and \eqref{16}, we can derive
\begin{equation}\label{17}
\begin{split}
&\frac{1}{(1-\lambda)^n}xH_{n-1} ^{(n)} (x|\lambda)\\
=&x\left(\frac{t\left(\frac{e^t-\lambda}{1-\lambda}\right)^a}{t(e^t-\lambda)}\right)^nx^{-1}xH_{n-1} ^{(an)} (x|\lambda)\\
=&\frac{x}{(1-\lambda)^{an}}\left(e^t-\lambda\right)^{(a-1)n}H_{n-1} ^{(an)}(x|\lambda) \\
=&\frac{x}{(1-\lambda)^{an}}\sum_{l=0} ^{(a-1)n}\binom{(a-1)n}{l}(-\lambda)^{(a-1)n-l}e^{lt}H_{n-1} ^{(an)}(x|\lambda)\\
=&\frac{x}{(1-\lambda)^{an}}\sum_{l=0} ^{(a-1)n}\binom{(a-1)n}{l}(-\lambda)^{(a-1)n-l}H_{n-1} ^{(an)}(x+l|\lambda),\\
\end{split}
\end{equation}
where $a,n \in {\mathbb{N}}$. Therefore, by \eqref{17}, we obtain the following theorem.
\begin{theorem}\label{thm1}
For $a,n\in {\mathbb{N}}$, we have
\begin{equation*}
H_{n-1} ^{(n)} (x|\lambda)=\frac{1}{(1-\lambda)^{(a-1)n}}\sum_{l=0} ^{(a-1)n}\binom{(a-1)n}{l}(-\lambda)^{(a-1)n-l}H_{n-1} ^{(an)}(x+l|\lambda).
\end{equation*}
\end{theorem}
Let us consider the following associated sequences:
\begin{equation}\label{18}
\frac{1}{(1-\lambda)^n}xH_{n-1} ^{(n)} (x|\lambda)\sim\left(1,t(e^t-\lambda)\right),~p_n(x)\sim\left(1,\left(\frac{1-\lambda}{e^t-\lambda}\right)^at\right),~(a \neq 0).
\end{equation}
For $x^n \sim (1,t)$, by \eqref{10} and \eqref{18}, we get
\begin{equation}\label{19}
\begin{split}
p_n(x)&=x\left(\frac{t}{t\left(\frac{1-\lambda}{e^t-\lambda}\right)^a}\right)^nx^{-1}x^n=x\left(\frac{e^t-\lambda}{1-\lambda}\right)^{an}x^{n-1} \\
&=x\frac{1}{(1-\lambda)^{an}}\sum_{l=0} ^{an}\binom{an}{l}(-\lambda)^{an-l}(x+l)^{n-1}.
\end{split}
\end{equation}
For $n \geq 1$, by \eqref{10} and \eqref{18}, we get
\begin{equation}\label{20}
\begin{split}
p_n(x)&=x\left(\frac{t\left(e^t-\lambda\right)}{t\left(\frac{1-\lambda}{e^t-\lambda}\right)^a}\right)^nx^{-1}\frac{x}{(1-\lambda)^n}H_{n-1} ^{(n)} (x|\lambda) \\
&=x\left(\frac{1}{1-\lambda}\right)^{(a+1)n}\left(e^t-\lambda\right)^{(a+1)n}H_{n-1} ^{(n)}(x|\lambda).
\end{split}
\end{equation}
By \eqref{19} and \eqref{20}, we get
\begin{equation}\label{21}
\begin{split}
&\sum_{l=0} ^{an}\binom{an}{l}(-\lambda)^{an-l}(x+l)^{n-1} \\
=&\frac{1}{(1-\lambda)^n}\left(e^t-\lambda\right)^{(a+1)n}H_{n-1} ^{(n)}(x|\lambda) \\
=&\frac{1}{(1-\lambda)^n}\sum_{l=0} ^{(a+1)n}\binom{(a+1)n}{l}(-\lambda)^{(a+1)n-l}H_{n-1} ^{(n)} (x+l|\lambda).
\end{split}
\end{equation}
Therefore, by \eqref{21}, we obtain the following theorem.
\begin{theorem}\label{thm2}
For $n \geq 1$ and $a \in {\mathbb{Z}}_+={\mathbb{N}}\cup\left\{0\right\}$, we have
\begin{equation*}
\sum_{l=0} ^{an}\binom{an}{l} \left(-\lambda\right)^{-l}(x+l)^{n-1}= \frac{1}{(1-\lambda)^n}\sum_{l=0} ^{(a+1)n}\binom{(a+1)n}{l}(-\lambda)^{n-l}H_{n-1} ^{(n)} (x+l|\lambda).
\end{equation*}
\end{theorem}
Let us consider the following associated sequences:
\begin{equation}\label{22}
(x)_n\sim\left(1,e^t-1\right),~xH_{n-1} ^{(an)}(x|\lambda)\sim\left(1,t\left(\frac{e^t-\lambda}{1-\lambda}\right)^a\right),~(a \neq 0).
\end{equation}
By \eqref{10} and \eqref{22}, we get
\begin{equation}\label{23}
\begin{split}
xH_{n-1} ^{(an)}(x|\lambda)&=x\left(\frac{e^t-1}{t\left(\frac{e^t-\lambda}{1-\lambda}\right)^a}\right)^nx^{-1}(x)_n \\
&=x\left(\frac{e^t-1}{t}\right)^n\left(\frac{1-\lambda}{e^t-\lambda}\right)^{an}(x-1)_{n-1}.
\end{split}
\end{equation}
Replacing $x$ by $x+1$, we have
\begin{equation}\label{24}
\begin{split}
H_{n-1} ^{(an)}(x+1|\lambda)&=\left(\frac{e^t-1}{t}\right)^n\left(\frac{1-\lambda}{e^t-\lambda}\right)^{an}\sum_{l=0} ^{n-1}S_1(n-1,l)x^l \\
&=\left(\frac{e^t-1}{t}\right)^n\sum_{l=0} ^{n-1}S_1(n-1,l)H_l ^{(an)}(x|\lambda) \\
&=\sum_{l=0} ^{n-1} \sum_{k=0} ^lS_1(n-1,l)\frac{n!}{(k+n)!}S_2(k+n,n)(l)_kH_{l-k} ^{(an)}(x|\lambda) \\
&=\sum_{l=0} ^{n-1} \sum_{k=0} ^lS_1(n-1,l)S_2(k+n,n)\frac{\binom{l}{k}}{\binom{k+n}{n}}H_{l-k} ^{(an)} (x|\lambda).
\end{split}
\end{equation}
Therefore, by \eqref{24}, we obtain the following theorem.
\begin{theorem}\label{thm3}
For $n \geq 1$, $a \in {\mathbb{Z}}_+$, we have
\begin{equation*}
H_{n-1} ^{(an)}(x+1|\lambda)=\sum_{l=0} ^{n-1} \sum_{k=0} ^lS_1(n-1,l)S_2(k+n,n)\frac{\binom{l}{k}}{\binom{k+n}{n}}H_{l-k} ^{(an)} (x|\lambda)
\end{equation*}
\end{theorem}
Let
\begin{equation}\label{25}
\begin{array}{cc}
xH_{n-1} ^{(an)}(x|\lambda)\sim\left(1,t\left(\frac{e^t-\lambda}{1-\lambda}\right)^a\right),~(a \neq 0), \\
(x)_n\sim\left(1,e^t-1\right).
\end{array}
\end{equation}
Then, by \eqref{10} and \eqref{25}, we get
\begin{equation}\label{26}
\begin{split}
(x)_n&=x\left(\frac{t\left(\frac{e^t-\lambda}{1-\lambda}\right)^a}{e^t-1}\right)^nx^{-1}xH_{n-1} ^{(an)} (x|\lambda) \\
&=x\left(\frac{t}{e^t-1}\right)^n\left(\frac{e^t-\lambda}{1-\lambda}\right)^{an}H_{n-1} ^{(an)}(x|\lambda) \\
&=x\left(\frac{t}{e^t-1}\right)^nx^{n-1}\\
&=xB_{n-1} ^{(n)} (x).
\end{split}
\end{equation}
and
\begin{equation}\label{27}
(x)_n=\sum_{l=0} ^n S_1(n,l)x^l=x\sum_{l=0} ^{n-1}S_1(n,l+1)x^l,~(n\geq 1).
\end{equation}
Therefore, by \eqref{26} and \eqref{27}, we get
\begin{theorem}\label{thm4}
For $n\geq 1$, $0 \leq l \leq n-1$, we have
\begin{equation*}
S_1(n,l+1)=\binom{n-1}{l}B_{n-1-l} ^{(n)}.
\end{equation*}
\end{theorem}
From \eqref{26}, we note that
\begin{equation}\label{28}
\left(\frac{e^t-1}{t}\right)^n(x-1)_{n-1}=(1-\lambda)^{-an}(e^t-\lambda)^{an}H_{n-1} ^{(an)}(x|\lambda),~(n \geq 1).
\end{equation}
\begin{equation}\label{29}
\begin{split}
LHS{\text{ of }}\eqref{28}&=\left(\frac{e^t-1}{t}\right)^n\sum_{l=0} ^{n-1} S_1(n-1,l),(x-1)^l \\
&=\sum_{l=0} ^{n-1}S_1(n-1,l)\sum_{k=0} ^l\frac{n!l!}{(k+n)!(l-k)!}S_2(k+n,n)(x-1)^{l-k} \\
&=\sum_{l=0} ^{n-1}S_1(n-1,l)\sum_{k=0} ^l\frac{\binom{l}{k}}{\binom{k+n}{n}}S_2(k+n,n)(x-1)^{l-k},
\end{split}
\end{equation}
and
\begin{equation}\label{30}
\begin{split}
RHS{\text{ of }}\eqref{28}&=(1-\lambda)^{-an}\sum_{l=0} ^{an}\binom{an}{l}(-\lambda)^{an-l}e^{lt}H_{n-1} ^{(an)} (x|\lambda)\\
&=(1-\lambda)^{-an}\sum_{l=0} ^{an}\binom{an}{l}(-\lambda)^{an-l}H_{n-1} ^{(an)}(x+l|\lambda).
\end{split}
\end{equation}
Therefore, by \eqref{28}, \eqref{29} and \eqref{30}, we obtain the following theorem.
\begin{theorem}\label{thm5}
For $n \geq 1$, $a \in {\mathbb{Z}}_+$, we have
\begin{equation*}
\begin{split}
&(1-\lambda)^{-an}\sum_{l=0} ^{an}\binom{an}{l}(-\lambda)^{an-l}H_{n-1} ^{(an)}(x+l|\lambda) \\
=&\sum_{l=0} ^{n-1}\sum_{k=0} ^l\frac{\binom{l}{k}}{\binom{k+n}{n}}S_1(n-1,l)S_2(k+n,n)(x-1)^{l-k}.
\end{split}
\end{equation*}
\end{theorem}
Let
\begin{equation}\label{31}
p_n(x)\sim\left(1,\left(\frac{1-\lambda}{e^t-\lambda}\right)^at\right),~(x)_n\sim\left(1,e^t-1\right),~(a \neq 0).
\end{equation}
By \eqref{19}, we have
\begin{equation}\label{32}
\begin{split}
p_n(x)
&=\left(\frac{1}{1-\lambda}\right)^{an}x\sum_{l=0} ^{an}\binom{an}{l}(-\lambda)^{an-l}(x+l)^{n-1} \\
&=\left(\frac{1}{1-\lambda}\right)^{an}x\sum_{k=0} ^{an}\binom{an}{k}(-\lambda)^{an-k}\sum_{l=0} ^{n-1} \binom{n-1}{l}k^{n-1-l}x^l.
\end{split}
\end{equation}
From \eqref{10} and \eqref{31}, we have
\begin{equation}\label{33}
\begin{split}
&(x)_n =x\left(\frac{t\left(\frac{1-\lambda}{e^t-\lambda}\right)^a}{e^t-1}\right)^nx^{-1}p_n(x) \\
=&x\left(\frac{t}{e^t-1}\right)^n\left(\frac{1-\lambda}{e^t-\lambda}\right)^{an}\left(\frac{1}{1-\lambda}\right)^{an}\sum_{k=0} ^{an}\sum_{l=0} ^{n-1}\binom{an}{k}\binom{n-1}{l}k^{n-1-l}(-\lambda)^{an-k}x^l \\
=&\left(\frac{1}{1-\lambda}\right)^{an}x\left(\frac{t}{e^t-1}\right)^n\sum_{k=0} ^{an}\sum_{l=0} ^{n-1}\binom{an}{k}\binom{n-1}{l}k^{n-1-l}(-\lambda)^{an-k}H_l ^{(an)}(x|\lambda)\\
=&\left(\frac{1}{1-\lambda}\right)^{an}x\sum_{k=0} ^{an}\sum_{l=0} ^{n-1}\sum_{m=0} ^l \binom{an}{k}\binom{n-1}{l}\binom{l}{m}k^{n-1-l}(-\lambda)^{an-k}H_{l-m} ^{(an)}(\lambda)B_m ^{(n)}(x)\\
=&\left(\frac{1}{1-\lambda}\right)^{an}x\sum_{k=0} ^{an}\sum_{l=0} ^{n-1}\sum_{m=0} ^l\sum_{p=0} ^m \binom{an}{k}\binom{n-1}{l}\binom{l}{m}\binom{m}{p}k^{n-1-l}(-\lambda)^{an-k}H_{l-m} ^{(an)}(\lambda)B_{m-p} ^{(n)}x^p\\
=&\left(\frac{1}{1-\lambda}\right)^{an}x\sum_{p=0} ^{n-1}\left\{\sum_{k=0} ^{an}\sum_{l=p} ^{n-1}\sum_{m=p} ^l\binom{an}{k}\binom{n-1}{l}\binom{l}{m}\binom{m}{p}k^{n-1-l}(-\lambda)^{an-k}H_{l-m} ^{(an)}(\lambda)B_{m-p} ^{(n)}\right\}x^p.
\end{split}
\end{equation}
Therefore, by \eqref{27} and \eqref{33}, we obtain the following theorem.
\begin{theorem}\label{thm6}
For $n \geq 1$, $a \in {\mathbb{Z}}_+$ and $0 \leq p \leq n-1$, we have
\begin{equation*}
S_1(n,p+1)=\left(\frac{1}{1-\lambda}\right)^{an}\sum_{k=0} ^{an}\sum_{l=p} ^{n-1}\sum_{m=p} ^l\binom{an}{k}\binom{n-1}{l}\binom{l}{m}\binom{m}{p}k^{n-1-l}(-\lambda)^{an-k}H_{l-m} ^{(an)}(\lambda)B_{m-p} ^{(n)}.
\end{equation*}
\end{theorem}

\begin{theorem}\label{thm8}
For $n \geq 0$, we have
\begin{equation*}
e^{-xt}(e^t-\lambda)^n=\sum_{k=0} ^{\infty}\left(\sum_{l=0} ^n \sum_{j=0} ^k \binom{n}{l}(1-\lambda)^{n-l}\frac{\binom{k}{j}}{\binom{j+l}{l}}S_2(j+l,l)(-1)^{k-j}x^{k-j}\right)\frac{t^{k+l}}{k!}.
\end{equation*}
\end{theorem}

\begin{proof}
Note that
\begin{equation}\label{36}
e^{-xt}(e^t-\lambda)^n=e^{-xt}(e^t-1+1-\lambda)^n=\sum_{l=0} ^n\binom{n}{l}(1-\lambda)^{n-l}(e^t-1)^le^{-xt},
\end{equation}
and
\begin{equation}\label{37}
\begin{split}
(e^t-1)^le^{-xt}&=\sum_{k=0} ^{\infty}\left(\sum_{j=0} ^k \frac{l!k!}{(j+l)!(k-j)!}S_2(j+l,l)(-1)^{k-j}x^{k-j}\right)\frac{t^{k+l}}{k!}\\
&=\sum_{k=0} ^{\infty}\left(\sum_{j=0} ^k \frac{\binom{k}{j}}{\binom{j+l}{l}}S_2(j+l,l)(-1)^{k-j}x^{k-j}\right)\frac{t^{k+l}}{k!}.
\end{split}
\end{equation}
From \eqref{36} and \eqref{37}, we can derive Theorem \ref{thm8}.
\end{proof}
By (1.12) and Theorem \ref{thm8}, we get
\begin{equation}\label{38}
\begin{split}
&A_n(x;b)=x(x-bn)^{n-1}=x\left(\frac{1}{1-\lambda}\right)^{an}e^{-nbt}(e^t-\lambda)^{an}H_{n-1} ^{(an)}(x|\lambda) \\
=&\frac{x}{(1-\lambda)^{an}}\sum_{l=0} ^{an}\sum_{k=0} ^{n-1-l} \sum_{j=0} ^k \frac{\binom{an}{l}\binom{k}{j}(n-1)_{k+l}}{\binom{j+l}{l}}(1-\lambda)^{an-l}S_2(j+l,l)(-1)^{k-j}(nb)^{k-j}H_{n-1-l-k} ^{(an)} (x|\lambda).
\end{split}
\end{equation}
Therefore, by \eqref{38}, we obtain the following theorem.
\begin{theorem}\label{thm9}
For $n \geq 1$, $a \in {\mathbb{Z}}_+$ and $b \neq 0$, we have
\begin{equation*}
(x-bn)^{n-1}=\sum_{l=0} ^{an}\sum_{k=0} ^{n-1-l} \sum_{j=0} ^k \frac{\binom{an}{l}\binom{k}{j}\frac{(n-1)_{k+l}}{k!}}{\binom{j+l}{l}(1-\lambda)^l}S_2(j+l,l)(-1)^{k-j}(nb)^{k-j}H_{n-1-l-k} ^{(an)} (x|\lambda).
\end{equation*}
\end{theorem}
Let us consider the Changhee polynomials of the second kind as follows:
\begin{equation}\label{39}
\sum_{k=0} ^{\infty}C_k(x|\lambda)\frac{t^k}{k!}=\frac{1}{1+\lambda(1+t)}(1+t)^x.
\end{equation}
From \eqref{9} and \eqref{39}, we note that
\begin{equation}\label{40}
C_k(x|\lambda)\sim\left(1+\lambda e^t,e^t-1\right).
\end{equation}
Hence $\lambda \in {\mathbb{C}}$ with $\lambda \neq -1$.
Thus, by \eqref{40}, we get
\begin{equation}\label{41}
(1+\lambda e^t)C_n(x|\lambda)=(x)_n\sim\left(1,e^t-1\right),
\end{equation}
and
\begin{equation}\label{42}
(x)_n=x\left(\frac{t}{e^t-1}\right)^nx^{-1}x^n=x\left(\frac{t}{e^t-1}\right)^nx^{n-1}=xB_{n-1} ^{(n)} (x).
\end{equation}
Thus, by \eqref{41} and \eqref{42}, we get
\begin{equation}\label{43}
\begin{split}
C_n(x|\lambda)&=\frac{1}{\lambda e^t+1}xB_{n-1} ^{(n)} (x)=\sum_{l=0} ^n (-\lambda)^le^{lt}\left(xB_{n-1} ^{(n)} (x)\right) \\
&=\sum_{l=0} ^n (-\lambda)^l(x+l)B_{n-1} ^{(n)} (x+l).
\end{split}
\end{equation}
Let
\begin{equation}\label{44}
t_n(x|\lambda)\sim\left(1,\frac{t}{1+\lambda(1+t)}\right).
\end{equation}
Then, by \eqref{10} and \eqref{44}, we get
\begin{equation}\label{45}
\begin{split}
t_n(x|\lambda)&=x\left(\frac{t}{\frac{t}{1+\lambda(1+t)}}\right)^nx^{-1}x^n=x(1+\lambda(1+t))^nx^{n-1}\\
&=x\sum_{l=0} ^n\binom{n}{l}\lambda^l(1+t)^lx^{n-1}=x\sum_{a=0} ^n\binom{n}{a}\lambda^a\sum_{b=0} ^{n-1}\binom{a}{b}t^bx^{n-1}\\
&=\sum_{a=0} ^n\sum_{b=0} ^{n-1}\binom{n}{a}\binom{a}{b}\lambda^a(n-1)_bx^{n-b}=\sum_{a=0} ^n\sum_{b=1} ^{n}\lambda^a\binom{n}{a}\binom{a}{n-b}(n-1)_{n-b}x^b\\
&=\sum_{a=0} ^n\sum_{b=1} ^n \lambda^a\binom{n}{a}\binom{a}{n-b}\frac{(n-1)!}{(b-1)!}x^b.
\end{split}
\end{equation}
Let us also consider the following associated sequence:
\begin{equation}\label{46}
S_n(x|\mu)\sim\left(1,\frac{t}{(1+t)^{\mu}}\right),~(\mu\in{\mathbb{N}}).
\end{equation}
Then, by \eqref{10} and \eqref{46}, we easily get
\begin{equation}\label{47}
S_n(x|\mu)=\sum_{k=1} ^n \binom{\mu n}{n-k}\frac{(n-1)!}{(k-1)!}x^k.
\end{equation}
From \eqref{10}, \eqref{45} and \eqref{46}, we can derive
\begin{equation}\label{48}
\begin{split}
&S_n(x|\mu)=x\left(\frac{\frac{t}{1+\lambda(1+t)}}{\frac{t}{(1+t)^{\mu}}}\right)^nx^{-1}t_n(x|\lambda)=x\left(\frac{(1+t)^{\mu}}{1+\lambda(1+t)}\right)^nx^{-1}t_n(x|\lambda)\\
&=x\left(\sum_{l=0} ^{\infty}\frac{C_l(\mu|\lambda)}{l!}t^l\right)^nx^{-1}t_n(x|\lambda)\\
&=x\sum_{l=0} ^{\infty}\left\{\sum_{l_1+\cdots+l_n=l}\binom{l}{l_1,\ldots,l_n}\left(\prod_{i=1} ^n C_{l_i}(\mu|\lambda)\right)\right\}\frac{t^l}{l!}x^{-1}t_n(x|\lambda)\\
&=x\sum_{l=0} ^{\infty}\left\{\sum_{l_1+\cdots+l_n=l}\binom{l}{l_1,\ldots,l_n}\left(\prod_{i=1} ^n C_{l_i}(\mu|\lambda)\right)\right\}\frac{t^l}{l!} \left\{\sum_{a=0} ^n \sum_{b=1} ^n \lambda^a\binom{n}{a}\binom{a}{n-b}\frac{(n-1)!}{(b-1)!}x^{b-1}\right\} \\
&=x\sum_{a=0} ^n\sum_{b=1} ^n\sum_{l=0} ^{b-1}\sum_{l_1+\cdots+l_n=l}\binom{l}{l_1,\ldots,l_n}\left(\prod_{i=1} ^n C_{l_i}(\mu|\lambda)\right)\lambda^a\binom{n}{a}\binom{a}{n-b}\frac{(n-1)!}{(b-1)!}\frac{(b-1)_l}{l!}x^{b-1-l} \\
&=\sum_{a=0} ^n\sum_{b=1} ^n\sum_{l=0} ^{b-1}\sum_{l_1+\cdots+l_n=l}\binom{l}{l_1,\ldots,l_n}\left(\prod_{i=1} ^n C_{l_i}(\mu|\lambda)\right)\lambda^a\binom{n}{a}\binom{a}{n-b}\frac{(n-1)!}{(b-1)!}\binom{b-1}{l}x^{b-l} \\
&=\sum_{k=1} ^n \left\{\sum_{a=0} ^n\sum_{b=k} ^n\sum_{l_1+\cdots+l_n=b-k}\binom{b-k}{l_1,\ldots,l_n}\left(\prod_{i=1} ^n C_{l_i}(\mu|\lambda)\right)\lambda^a\binom{n}{a}\binom{a}{n-b}\binom{b-1}{k-1}\frac{(n-1)!}{(b-1)!}\right\}x^k.
\end{split}
\end{equation}
Therefore, by \eqref{47} and \eqref{48}, we obtain the following theorem.
\begin{theorem}\label{thm10}
For $n \geq 1$, $1 \leq k \leq n$, $b \neq 0$ and $\mu,a\in{\mathbb{Z}}_+$, we have
\begin{equation*}
\frac{\binom{\mu n}{n-k}}{(k-1)!}=\sum_{a=0} ^n\sum_{b=k} ^n\sum_{l_1+\cdots+l_n=b-k}\binom{n}{a}\binom{a}{n-b}\binom{b-k}{l_1,\ldots,l_n}\binom{b-1}{k-1}\lambda^a\frac{1}{(b-1)!}\left(\prod_{i=1} ^n C_{l_i}(\mu|\lambda)\right).
\end{equation*}
\end{theorem}
{\scshape{Remark.}} From \eqref{1}, we note that
\begin{equation}\label{49}
\begin{split}
\frac{1-\lambda}{e^t-\lambda}&=\frac{1-\lambda}{e^t-1+1-\lambda}=\frac{1}{1+\frac{e^t-1}{1-\lambda}}=\sum_{l=0} ^{\infty}(-1)^l\left(\frac{e^t-1}{1-\lambda}\right)^l \\
&=\sum_{k=0} ^{\infty}\left(\sum_{l=0} ^k(-1)^l\left(\frac{1}{1-\lambda}\right)^ll!S_2(k,l)\right)\frac{t^k}{k!},
\end{split}
\end{equation}
and
\begin{equation}\label{50}
\frac{1-\lambda}{e^t-\lambda}=\sum_{n=0} ^{\infty}H_n(\lambda)\frac{t^n}{n!},
\end{equation}
where $H_n(\lambda)$ are the Frobenius-Euler numbers. By \eqref{49} and \eqref{50}, we get
\begin{equation}\label{51}
H_k(\lambda)=\sum_{l=0} ^k\left(\frac{1}{\lambda-1}\right)^ll!S_2(k,l).
\end{equation}
Let us consider the following associated sequences:
\begin{equation}\label{52}
p_n(x)\sim\left(1,\frac{1-\lambda}{e^t-\lambda}t\right),~x^n\sim(1,t).
\end{equation}
Then, by \eqref{10} and \eqref{52}, we get
\begin{equation}\label{53}
p_n(x)=x\left(\frac{1}{1-\lambda}\right)^n(e^t-\lambda)^nx^{n-1}=\left(\frac{1}{1-\lambda}\right)^nx\sum_{k=0} ^n \binom{n}{k}(-\lambda)^{n-k}(x+k)^{n-1},
\end{equation}
and
\begin{equation}\label{54}
x^n=x\left(\frac{\left(\frac{1-\lambda}{e^t-\lambda}\right)t}{t}\right)^nx^{-1}p_n(x)=x\left(\frac{1-\lambda}{e^t-\lambda}\right)^nx^{-1}p_n(x).
\end{equation}
Thus, by \eqref{53} and \eqref{54}, we get
\begin{equation*}
\begin{split}
&x^{n-1}=\sum_{l=0} ^{n-1}\left\{\sum_{l_1+\cdots+l_n=l}\binom{l}{l_1,\ldots,l_n}H_{l_1}(\lambda)\cdots H_{l_n}(\lambda)\right\}\frac{t^l}{l!}\times\left\{\sum_{k=0} ^n \binom{n}{k}(-\lambda)^{n-k} \frac{(x+k)^{n-1}}{(1-\lambda)^n}\right\} \\
&=\frac{1}{(1-\lambda)^n}\sum_{k=0} ^n \sum_{l=0} ^{n-1}\sum_{l_1+\cdots+l_n=l}\binom{l}{l_1,\ldots,l_n}\left(\prod_{i=1} ^n H_{l_i}(\lambda)\right)\binom{n}{k}\binom{n-1}{l}(-\lambda)^{n-k}(x+k)^{n-1-l}.
\end{split}
\end{equation*}

\end{document}